\documentclass[12pt]{article}

\usepackage[T1]{fontenc}
\usepackage{amsmath}
\usepackage{amssymb}
\usepackage{amsthm}
\usepackage{tikz-cd}
\usepackage{theoremref}
\usepackage[hidelinks]{hyperref}
\usepackage{cleveref}
\usepackage{enumitem}
\usepackage{abstract}
\usepackage[margin=3cm]{geometry}

\newtheorem{theorem}{Theorem}[section]
\newtheorem{proposition}{Proposition}[section]
\newtheorem{lemma}{Lemma}[section]
\newtheorem{corollary}{Corollary}[section]
\theoremstyle{definition}
\newtheorem{example}{Example}[section]
\newtheorem{remark}{Remark}[section]

\newcommand{\ce}{\mathcal{E}}
\newcommand{\cf}{\mathcal{F}}

\newcommand{\ci}{\mathcal{J}}
\newcommand{\ck}{\mathcal{K}}
\newcommand{\cl}{\mathcal{L}}

\newcommand{\co}{\mathcal{O}}

\newcommand{\cq}{\mathcal{Q}}
\newcommand{\cs}{\mathcal{S}}

\newcommand{\bp}{\mathbf{P}}

\renewcommand{\H}{\mathrm{H}}

\newcommand{\Ext}{\operatorname{Ext}}
\newcommand{\SExt}{\mathcal{E}xt}
\newcommand{\Hom}{\operatorname{Hom}}
\newcommand{\End}{\operatorname{End}}
\newcommand{\SHom}{\mathcal{H}om}
\newcommand{\Quot}{\operatorname{Quot}}
\newcommand{\Ker}{\operatorname{Ker}}
\newcommand{\Tr}{\operatorname{Tr}}
\newcommand{\vir}{\operatorname{vir}}

\begin{document}

\title{\Large\textbf{Cosection localization and the Quot scheme $\Quot^{l}_{S}(\ce)$}}
\author{\normalsize SAMUEL STARK}
\date{}
\maketitle

\begin{abstract}
Let $\ce$ be a locally free sheaf of rank $r$ on a smooth projective surface $S$. The Quot scheme $\Quot^{l}_{S}(\ce)$ of length $l$ coherent sheaf quotients of $\ce$ is 
a natural higher rank generalization of the Hilbert scheme of $l$ points of $S$. We study the virtual intersection theory of this scheme. If $C\subset S$ is a smooth canonical curve, we use cosection localization to show that the virtual fundamental class of $\Quot^{l}_{S}(\ce)$ is $(-1)^{l}$ times the fundamental class of the smooth subscheme $\Quot^{l}_{C}(\ce\vert_{C})\subset\Quot^{l}_{S}(\ce)$.
We then prove a structure theorem for virtual tautological integrals over $\Quot^{l}_{S}(\ce)$. From this we deduce, among other things, the equality of virtual Euler characteristics
$\chi^{\vir}(\Quot^{l}_{S}(\ce))=\chi^{\vir}(\Quot^{l}_{S}(\co^{\oplus r}))$.	
\end{abstract}

\section{Introduction}\label{sec1}

Let $S$ be a smooth projective complex surface. The Hilbert scheme of $l$ points $S^{[l]}$ is a canonical smooth compactification
of the configuration space of $l$ unordered points of $S$, which occurs naturally in several branches of mathematics. A point of $S^{[l]}$
corresponds to a zero-dimensional closed subscheme $Z\subset S$ of length $l$; since a subscheme can be specified by either the quotient $\co\rightarrow\co_{Z}$
or by its kernel (i.e. the ideal sheaf of $Z\subset S$), there are two natural generalizations of $S^{[l]}$:
\begin{enumerate}[label=(\arabic*)]
\item the Quot schemes  $\Quot^{l}_{S}(\ce)$ of length $l$ quotients of a locally free sheaf $\ce$ of higher rank $r$; \label{eins}
\item the moduli spaces of stable sheaves of higher rank. \label{zwei}
\end{enumerate}
Common to both higher rank generalizations is that they are typically singular, and both admit a virtual fundamental class. By integrating universal cohomology classes over it, one obtains invariants of these moduli spaces. For \ref{zwei}, these are the (algebraic) Donaldson invariants of $S$, which are well studied, see for instance \cite{Do, EG, FS, Go2, GK, KM, Li, Moc, OG}; the study of the invariants corresponding to \ref{eins} is the subject matter of this paper. The schemes $\Quot^{l}_{S}(\ce)$ satisfy three basic properties:
\begin{enumerate}[label=(\Roman*)]
\item a quotient $\ce\rightarrow\ce^{''}$ induces a closed immersion \label{one} $$\iota:\Quot^{l}_{S}(\ce^{''})\rightarrow\Quot^{l}_{S}(\ce),$$ 
\item an invertible sheaf $\cl$ on $S$ induces an isomorphism  \label{two} $$\Quot^{l}_{S}(\ce)\xrightarrow{\sim}\Quot^{l}_{S}(\ce\otimes\cl),$$
\item the automorphism group of $\ce$ acts on $\Quot^{l}_{S}(\ce)$.  \label{three}
\end{enumerate}
Moreover, a locally free sheaf $\cf$ on $S$ induces an associated tautological sheaf $\cf^{[l]}$ on $\Quot^{l}_{S}(\ce)$, which is compatible with \ref{one}-\ref{three};
we then obtain invariants of $\Quot^{l}_{S}(\ce)$ by taking integrals of Chern classes of tautological sheaves over the virtual fundamental class
$$[\Quot^{l}_{S}(\ce)]^{\vir}\in A_{lr}(\Quot^{l}_{S}(\ce)).$$
In the special case $\ce=\co^{\oplus r}$, the fixed locus of the action of the automorphism group of $\ce$  is given by products of Hilbert schemes of points of $S$. It is this observation which is the point of departure of the recent work of Oprea and Pandharipande \cite{OP},
which studies these integrals using \ref{three} and the virtual torus localization technique of \cite{GP}.

It is natural to study these invariants for arbitrary locally free sheaves $\ce$, by systematically exploiting all three properties
\ref{one}-\ref{three}. Our first result says that using property \ref{one},
the virtual fundamental class can be localized along the zero locus of a $2$-form.
\begin{theorem}[\Cref{explicitquot}]\label{canonical}
Let $C\subset S$ be a smooth curve cut out by a $2$-form. Then
\begin{equation*}
[\Quot^{l}_{S}(\ce)]^{\vir}=(-1)^{l}\iota_{\ast}[\Quot^{l}_{C}(\ce\vert_{C})].
\end{equation*}
\end{theorem}
Localizing along the zeros of a $2$-form is a well-known technique in the study of moduli spaces associated to surfaces \cite{KL, MP, VW,Wit}.
In the case $\ce=\co^{\oplus r}$ \Cref{canonical} was proven in \cite{OP} using torus localization;
in practice, this result allows one to compute tautological integrals over $[\Quot^{l}_{S}(\ce)]^{\vir}$ in terms of tautological integrals
over the smooth scheme $\Quot^{l}_{C}(\ce\vert_{C})$. It thus plays a key role in all the computations of \cite{OP} and those contained in the follow-up papers \cite{JOP,Lim,AJLOP}.
Our proof has three ingredients: for a $2$-form $\omega\in\H^{0}(\Omega^{2}_{S})$ with zero scheme $C$, we construct a cosection $\sigma(\omega)$ of the obstruction sheaf of $\Quot^{l}_{S}(\ce)$ with zero scheme $\Quot^{l}_{C}(\ce\vert_{C})$. By the work of Kiem and Li on cosection localization \cite{KL}, we obtain a $\sigma(\omega)$-localized virtual fundamental class
$$[\Quot^{l}_{S}(\ce)]_{\sigma(\omega)}^{\vir}\in A_{lr}(\Quot^{l}_{C}(\ce\vert_{C}))$$ 
such that $\iota_{\ast}[\Quot^{l}_{S}(\ce)]_{\sigma(\omega)}^{\vir}=[\Quot^{l}_{S}(\ce)]^{\vir}$.
We show
$$[\Quot^{l}_{S}(\ce)]_{\sigma(\omega)}^{\vir}=(-1)^{l}[\Quot^{l}_{C}(\ce\vert_{C})]$$
by using (a) the dimension of $\Quot^{l}_{C}(\ce\vert_{C})$ equals the virtual dimension of $\Quot^{l}_{S}(\ce)$, (b) the smooth locus of $\Quot^{l}_{S}(\ce)$ intersects the smooth locus of $\Quot^{l}_{C}(\ce\vert_{C})$. Our proof in particular answers a question raised by Oprea and Pandharipande \cite{OP}; it applies more generally to any scheme with a perfect obstruction theory and a cosection of its obstruction sheaf satisfying properties analogous to (a) and (b). Although cosection localization has found many applications in enumerative geometry \cite{Chang1,Chang2,KL2,Maulik}, primarily in combination with Graber-Pandharipande's virtual torus localization \cite{GP}, in the absence of a torus action, the cosection localized virtual fundamental class is poorly understood. In our situation we can compute it completely explicitly, which is fairly rare.

The fundamental observation which guides our approach is that property \ref{one} yields relations in the intersection theory of the schemes $\Quot^{l}_{S}(\ce)$.
Apart from \Cref{canonical}, another instance of this is the following theorem, which the author proved in \cite{Stark}. It provides a way of changing the sheaf $\ce$,
and as such it plays a key role in extending the results of \cite{OP} to arbitrary locally free sheaves.
\begin{theorem}\label{lf}
Let $\ce\rightarrow\ce^{''}$ be a locally free quotient with kernel $\ce^{'}$. Then
\begin{equation*}
\iota_{\ast}[\Quot^{l}_{S}(\ce^{''})]^{\vir}=e(\ce^{'\vee[l]})\cap[\Quot^{l}_{S}(\ce)]^{\vir}.
\end{equation*}
\end{theorem}

We observe that \Cref{canonical} in particular gives a simple proof of \Cref{lf} for surfaces $S$ containing a smooth canonical curve. By combining \Cref{canonical} and \Cref{lf}, we then obtain the following structure theorem for tautological integrals over $[\Quot^{l}_{S}(\ce)]^{\vir}$.
\begin{theorem}[\Cref{universality}]\label{universalityintro}
Let $P$ be a polynomial in the Chern classes of tautological sheaves $\cf_{1}^{[l]}, \ldots, \cf^{[l]}_m$ on $\Quot^{l}_{S}(\ce)$, and the virtual tangent bundle $T^{\vir}_{\Quot^{l}_{S}(\ce)}$.
Then the integral $$\int_{[\Quot^{l}_{S}(\ce)]^{\vir}} P$$ is given by a polynomial, with coefficients depending on $\ce,\cf_{1},\ldots,\cf_{m}$ only through their ranks,
in the intersection numbers $$c_{1}(\ce\otimes\cf_{1})c_{1}(S), \ldots, c_{1}(\ce\otimes\cf_{m})c_{1}(S), c_{1}(\ce)c_{1}(S), c_{1}(S)^{2}.$$
\end{theorem}
This is a virtual, higher rank analogue of a result of Ellingsrud, G\"ottsche, and Lehn \cite{EGL}, which has proven to be very useful in a number of contexts. The idea of the proof is as follows: by \ref{two} we can assume that $\ce$ is globally generated; \Cref{lf} allows us reduce to the case $\ce=\co^{\oplus r}$, in which case we can use \ref{three}, torus localization and the result of \cite{EGL}. Finally, we use \Cref{canonical} to eliminate second Chern classes, and \ref{two} to obtain the dependence through the tensor products.
Combined with the work of Oprea and Pandharipande, our theorems lead to a fairly complete understanding of the virtual intersection theory of $\Quot^{l}_{S}(\ce)$.
As an application of \Cref{universalityintro}, we prove the equality of virtual Euler characteristics
\begin{equation}\label{eq:euler}
\chi^{\vir}(\Quot^{l}_{S}(\ce))=\chi^{\vir}(\Quot^{l}_{S}(\co^{\oplus r})),
\end{equation}
using \ref{two}. The computation of $\chi^{\vir}(\Quot^{l}_{S}(\co^{\oplus r}))$ is one of the main results of \cite{OP}. We also deduce the symmetry property
\begin{equation}\label{eq:symmetry}
S_{\ce}^{\cf}=S^{\ce}_{\cf}
\end{equation}
of the virtual Segre series
$$S_{\ce}^{\cf}(q)=\sum_{l=0}^{\infty} (-1)^{lr}q^{l}\int_{[\Quot^{l}_{S}(\ce)]^{\vir}} s(\cf^{[l]})$$
from \Cref{universalityintro} and the result of \cite{AJLOP} in the $\ce=\co^{\oplus r}$ case. As shown in \cite{Bojko}, (\ref{eq:euler}) and (\ref{eq:symmetry})
can also be deduced from Joyce's wall-crossing formula \cite{Joyce}. However, our approach is more geometric and natural, and does not rely on the deep results of
\cite{Joyce}.  We also show that for any invertible sheaf $\cl$ on $S$
 \begin{equation*}
\sum_{l=0}^{\infty}q^{l}\int_{[\Quot^{l}_{S}(\ce)]^{\vir}}e(\cl^{[l]})^{r}=\left(\frac{1}{1-q}\right)^{c_{1}(\ce\otimes\cl)c_{1}(S)},
 \end{equation*}
 using \Cref{universalityintro} and \Cref{lf}.
 
\section{Cosection localization}

\subsection{Localized Gysin map}\label{gysindef}
Consider a scheme $X$, and a cosection $\sigma:\cf\rightarrow\co$ of a locally free sheaf $\cf$ on $X$. Let $Z(\sigma)$ denote the zero scheme of $\sigma$, and $F\rightarrow X$
the vector bundle associated to $\cf$, and write $F(\sigma)$ for the union $F\vert_{Z(\sigma)}\cup\Ker(\sigma)$, endowed with the reduced subscheme structure.
\begin{theorem}[Kiem-Li \cite{KL}]
	There exists  a $\sigma$-localized Gysin map
	\begin{equation*}
		0^{!}_{\sigma}:A_{\ast}(F(\sigma))\rightarrow A_{\ast}(Z(\sigma))
	\end{equation*}
	which is compatible with the usual Gysin map $0^{!}_{F}$ of $F$,
	in the sense that the diagram
	\begin{equation}\label{eq:compatibility}
		\begin{tikzcd}
			A_{\ast}(F(\sigma)) \arrow[r] \arrow[d, swap, "0^{!}_{\sigma}"] & A_{\ast}(F) \arrow[d, "0^{!}_{F}"]  \\
			A_{\ast}(Z(\sigma)) \arrow[r] & A_{\ast}(X)
		\end{tikzcd}
	\end{equation}
	commutes.
\end{theorem}
More precisely, Kiem and Li construct $0^{!}_{\sigma}$ as follows. If $C\subset F(\sigma)$ is a closed integral subscheme, 
then either $C$ is contained in $F\vert_{Z(\sigma)}$ or not; in the former case one simply puts 
\begin{equation*}
	0^{!}_{\sigma}([C])=0^{!}_{F\vert_{Z(\sigma)}}([C]).
\end{equation*}
In the latter case one considers the blow up $\tilde{X}=\mathrm{Bl}_{Z(\sigma)} X$, forms the Cartesian diagram 
\begin{equation}
	\begin{tikzcd}
		\tilde{F} \arrow[r] \arrow[d] & F \arrow[d] \\
		\tilde{X}  \arrow[r, swap, "\rho"] & X
	\end{tikzcd}
\end{equation}
and defines $\tilde{C}$ to be the proper transform of $C$. Let $\ci$ be the ideal sheaf of $Z(\sigma)$ in $X$, 
and write $D=\rho^{-1}(Z(\sigma))$ for the exceptional divisor of $\rho$, with inclusion $i:D\rightarrow\tilde{X}$.
Moreover, let $K$ be the vector bundle associated to the kernel of the morphism of sheaves
\begin{equation*}
	\rho^{\ast} \cf\xrightarrow{\rho^{\ast} \sigma}\rho^{\ast}\ci\rightarrow\ci_{D}=\co_{\tilde{X}}(-D).
\end{equation*}
The proper transform $\tilde{C}$ is a closed integral subscheme of $K$,
and $0^{!}_{\sigma}([C])$ is defined to be the image of $[\tilde{C}]$ under the composition
\begin{equation*}
	A_\ast(K)\xrightarrow{0_{K}^{!}}A_{\ast}(\tilde{X})\xrightarrow{-i^{!}}A_{\ast}(D)\xrightarrow{\rho_{D\ast}}A_{\ast}(Z(\sigma)).
\end{equation*}

\begin{remark}\label{Gysincompatibility}
	It is clear that the $\sigma$-localized Gysin map is compatible with open immersions $U\rightarrow X$,
	in the sense that the diagram 
	\begin{equation*}
		\begin{tikzcd}
			A_{\ast}(F(\sigma)) \arrow[r] \arrow[d] & A_{\ast}(F\vert_{U}(\sigma\vert_{U})) \arrow[d]   \\
			A_{\ast}(Z(\sigma)) \arrow[r] & A_{\ast}(Z(\sigma\vert_{U}))
		\end{tikzcd}
	\end{equation*}
	is commutative.
\end{remark}

\subsection{Localized virtual fundamental class}\label{localiseddef}

Let $X$ be a projective scheme with cotangent complex $L_{X}$. Recall \cite{BF} that a perfect obstruction theory on $X$ is a morphism
$\phi:E\rightarrow L_{X}$ in the derived category of $X$, where $E$ is a perfect $2$-term complex concentrated in degrees $-1$ and $0$,
such that $h^{0}(\phi)$ is an isomorphism and $h^{-1}(\phi)$ is surjective. We write $\co b=h^{1}(E^{\vee})$ for the obstruction sheaf,
and $v=\mathrm{rk}(E)$ for the virtual dimension. Behrend and Fantechi show that any perfect obstruction theory on $X$ yields a virtual fundamental
class on $X$, which can be constructed as follows \cite{GP}. If $[\cf^{-1}\rightarrow\cf^0]$ is a global resolution of $E$, then by considering the mapping cone of (a representative of) the truncation
\begin{equation*}
	E\rightarrow L_{X}\rightarrow\tau^{\geq-1}L_{X},
\end{equation*}
one obtains a closed subcone $C\subset F$, where $F$ denotes the vector bundle on $X$ associated to $(\cf^{-1})^{\vee}$.
The virtual fundamental class of $X$ is defined by
$$[X]^{\vir}=0^{!}_{F}[C]\in A_{v}(X),$$
which is independent of the choice of resolution of $E$.
\begin{example}\label{smooth}
If $X$ is smooth, then $[X]^{\vir}=e(\co b)\cap[X]$.
\end{example}
Consider now a cosection  $\sigma:\co b\rightarrow\co$ of the obstruction sheaf, and denote by $\iota:Z(\sigma)\rightarrow X$ the inclusion of the zero scheme of $\sigma$. By abuse of notation, we also write $\sigma$ for the induced cosection $(\cf^{-1})^{\vee}\rightarrow\co b\rightarrow\co$. Kiem and Li (\cite{KL}, Corollary 4.6) prove that cone $C\subset F$ is in fact  set theoretically contained in $F(\sigma)$, 
and therefore defines a class $[C]\in A_{\ast}(F(\sigma))$. Kiem-Li's $\sigma$-localized virtual fundamental class of $X$ is then given by
$$[X]^{\vir}_{\sigma}=0^{!}_{\sigma}[C]\in A_{v}(Z(\sigma)).$$
The diagram (\ref{eq:compatibility}) shows that  $[X]^{\vir}=\iota_{\ast}[X]^{\vir}_{\sigma}$.
\begin{remark}\label{open}
	Let $U\subset X$ be an open subscheme, with the induced perfect obstruction theory. Then the restriction map $A_{v}(Z(\sigma))\rightarrow A_{v}(Z(\sigma\vert_{U}))$
	takes $[X]^{\vir}_{\sigma}$ to $[U]^{\vir}_{\sigma\vert_{U}}$. Indeed, the cone $C\subset F$ over $X$ restricts to the corresponding cone over $U$, and so this
	follows from \Cref{Gysincompatibility}.
\end{remark}

\subsection{Explicit computations}
We now describe a situation in which the cosection localized virtual fundamental class can be computed explicitly.
In the following, we write $X^{0}$ for the smooth locus of a scheme $X$.

\begin{proposition}\label{explicit}
	Let $X$ be a projective scheme with a perfect obstruction theory of virtual dimension $v$,
	and let $\sigma:\co b\rightarrow\co$ be a cosection of the obstruction sheaf. Assume that the zero scheme $Z(\sigma)$ is irreducible of dimension $v$,
	and that $Z(\sigma)^{0}\cap X^{0}$ is nonempty. Then
	$$[X]^{\vir}_{\sigma}=(-1)^c [Z(\sigma)] \quad \mathrm{in} \quad A_{v}(Z(\sigma)),$$
	where $c$ is the codimension of $Z(\sigma)$ in $X$.
\end{proposition}

\begin{proof}
	Since $Z(\sigma)$ is irreducible of dimension $v$, we have 
	\begin{equation}\label{eq:multiplicity}
		[X]^{\vir}_{\sigma}=n[Z(\sigma)]
	\end{equation}
	for some integer $n$. Let $\sigma^{0}$ be the restriction of $\sigma$ to the smooth locus $X^{0}$ of $X$.
	By assumption $Z(\sigma^{0})=Z(\sigma)\cap X^{0}$ is nonempty, and applying \Cref{open} to (\ref{eq:multiplicity}) gives
	\begin{equation*}
		[X^{0}]^{\vir}_{\sigma^{0}}=n[Z(\sigma^{0})] \quad \mathrm{in} \quad A_{v}(Z(\sigma^{0})),
	\end{equation*}
	In fact, since $Z(\sigma)^{0}\cap X^{0}$ is nonempty, we can also replace $Z(\sigma)$ by $Z(\sigma)^{0}$, and thus assume that both $Z(\sigma)$ and $X$ are smooth.
	In this situation we show, by an explicit computation, that the coefficient in (\ref{eq:multiplicity}) is
	$$n=(-1)^c,$$
	where $c$ is the codimension of $Z(\sigma)$ in $X$.
	In fact more generally, dropping the assumption that the dimension of $Z(\sigma)$ is the virtual dimension of $X$, we show that 
	\begin{equation}\label{eq:smoothcase}
		[X]^{\vir}_{\sigma}= (-1)^{c}e(\Ker\sigma\vert_{Z(\sigma)})\cap [Z(\sigma)] \quad \mathrm{in} \quad A_{v}(Z(\sigma)),
	\end{equation}
	where $c$ is the codimension of $Z(\sigma)$ in $X$. Since $X$ is smooth, the obstruction sheaf $\co b$ is locally free of rank $l=\dim(X)-v$,
	and the kernel of $\sigma\vert_{Z(\sigma)}$ is locally free of rank $\dim Z(\sigma)-v$. In the notation of \ref{localiseddef}, the cone $C\subset F$
	is the the vector bundle associated to the kernel of the surjection $(\cf^{-1})^{\vee}\rightarrow\co b$. We now consider
        the construction of the $\sigma$-localized Gysin map, which we have described in detail in \ref{gysindef}. In the notation of \ref{gysindef},
         $\tilde{C}$ is in our situation a subbundle of $K$, so 
	\begin{equation*}
		0^{!}_{K} [\tilde{C}]=\left\{c(K)\cap s(\tilde{C})\right\}_{v+1}=c_{l-1}(K/\tilde{C})\cap [\tilde{X}]
	\end{equation*}
	by a standard result in intersection theory \cite{Fulton}, and in particular
	$$-\iota^\ast 0^{!}_{K}[\tilde{C}]=-c_{l-1}((K/\tilde{C})\vert_D)\cap[D].$$
	The inclusion $K\rightarrow\tilde{F}$ induces a morphism of exact sequences
	\begin{equation*}
		\begin{tikzcd}
			0\arrow[r] & \tilde{C}\vert_D \arrow[r] \arrow[d] & K\vert_D \arrow[r] \arrow[d] & (K/\tilde{C})\vert_D \arrow[r] \arrow[d] & 0  \\
			0\arrow[r] & \tilde{F}\vert_D \arrow[r] & \tilde{F}\vert_D \arrow[r] & 0 \arrow[r] & 0
		\end{tikzcd}
	\end{equation*}
	whose associated exact sequence of kernels and cokernels is of the form
	\begin{equation*}
		0\rightarrow (K/\tilde{C})\vert_D\rightarrow \rho_{D}^\ast(\co b\vert_{Z(\sigma)})\rightarrow\co_D(1)\rightarrow 0,
	\end{equation*}
	since $C$ is the the vector bundle associated to the kernel of $(\cf^{-1})^{\vee}\rightarrow\co b$, $\tilde{C}=\rho^{\ast}C\subset\tilde{F}=\rho^{\ast}F$, and
        $\co_{D}(1)=\co_{\tilde{X}}(-D)\vert_{D}$. Hence
	\begin{equation*}
		-\iota^\ast 0^{!}_{K} [\tilde{C}] =-\left\lbrace c(\rho_{D}^\ast(\co b\vert_{Z(\sigma)}))\cap s(\co_D(1))\right\rbrace_{l-1}\cap[D].
	\end{equation*}
	By the equation 
	\begin{equation*}
		\sum_{i\geqslant 0} \rho_{D\ast}(c_1(\co_{D}(1))^{i}\cap[D])=s(N_{Z(\sigma)/X})\cap[Z(\sigma)]
	\end{equation*}
	and the projection formula, we obtain
	\begin{align*}
		[X]^{\vir}_{\sigma} &= \sum_{i=0}^{l-1} (-1)^{i+1} c_{l-1-i}(\co b\vert_{Z(\sigma)})\cap \rho_{D\ast}(c_1(\co_D(1))^{i}\cap[D]) \\
		&=(-1)^{c}\left(\sum_{j=0}^{l-c} c_{l-c-j}(\co b\vert_{Z(\sigma)})(-1)^{j} s_{j}(N_{Z(\sigma)/X})\right) \cap[Z(\sigma)] \\
		&=(-1)^{c}e(\Ker\sigma\vert_{Z(\sigma)})\cap [Z(\sigma)]. 
	\end{align*}
      \end{proof}
      
\section{Virtual intersection theory}

\subsection{General properties}
Let $\ce$ be a coherent sheaf on a smooth projective surface $S$.
The Quot scheme $\Quot^{l}_{S}(\ce)$ is the moduli space of length $l$ coherent sheaf quotients of $\ce$. A point $q$ of $\Quot^{l}_{S}(\ce)$ thus corresponds to a quotient
\begin{equation*}
\ce\rightarrow\cq_{q} \quad \mathrm{on} \quad S, \quad \mathrm{with} \quad \dim\cq_{q}=0 \quad \mathrm{and} \quad h^{0}(\cq_{q})=l.
\end{equation*}
It satisfies the following basic properties \cite{Grothendieck}:
\begin{enumerate}[label=(\Roman*)]
	\item a quotient $\ce\rightarrow\ce^{''}$ induces a closed immersion \label{onea} $$\iota:\Quot^{l}_{S}(\ce^{''})\rightarrow\Quot^{l}_{S}(\ce),$$ 
	\item an invertible sheaf $\cl$ on $S$ induces an isomorphism  \label{twoa} $$\Quot^{l}_{S}(\ce)\xrightarrow{\sim}\Quot^{l}_{S}(\ce\otimes\cl),$$
	\item the automorphism group of $\ce$ acts on $\Quot^{l}_{S}(\ce)$.  \label{threea}
\end{enumerate}
We denote by 
$$\pi:S\times\Quot^{l}_{S}(\ce)\rightarrow\Quot^{l}_{S}(\ce)$$
the projection to $\Quot^{l}_{S}(\ce)$, by $\ce\rightarrow\cq$ the universal quotient\footnote{Here and elsewhere we will sometimes suppress notationally the pullback along projections.} on $S\times\Quot^{l}_{S}(\ce)$, and by $\cs$ its kernel. We will frequently take the Fourier-Mukai transform with kernel $\cq$:
a locally free sheaf $\cf$ on $S$ induces a tautological sheaf
$$\cf^{[l]}=\pi_{\ast}(\cf\otimes\cq) \quad \mathrm{on} \quad \Quot_{S}^{l}(\ce).$$
It is locally free with fibre $\H^{0}(\cf\otimes\cq_q)$ over a point $q$ corresponding to the quotient $\ce\rightarrow\cq_q$ on $S$.
An elementary but important observation is \cite{Stark}:
\begin{lemma}\label{compatibility}
	The sheaves $\cf^{[l]}$ are compatible with \ref{onea}-\ref{threea}.
\end{lemma}
More precisely, the sheaf $\cf^{[l]}$ on $\Quot^{l}_{S}(\ce)$ restricts to the sheaf $\cf^{[l]}$ on $\Quot^{l}_{S}(\ce^{''})$ under the embedding \ref{onea},
$\cf^{[l]}$ on $\Quot^{l}_{S}(\ce\otimes\cl)$ corresponds to $(\cf\otimes\cl)^{[l]}$ on $\Quot^{l}_{S}(\ce)$ under the isomorphism \ref{twoa},
and the sheaf $\cf^{[l]}$ on $\Quot^{l}_{S}(\ce)$ has a canonical linearization with respect to the action of the automorphism group of $\ce$ \ref{threea}.

In all that follows, we will take $\ce$ to be locally free of rank $r$. In this case, it is well-known that the canonical deformation-obstruction theory
\begin{equation*}
	R\SHom_{\pi}(\cs,\cq)^{\vee}\rightarrow L_{\Quot^{l}_{S}(\ce)}
\end{equation*}
is perfect of virtual dimension $lr$ \cite{OP}, and therefore gives rise to a virtual fundamental class 
$$[\Quot^{l}_{S}(\ce)]^{\vir}\in A_{lr}(\Quot^{l}_{S}(\ce)).$$
We will denote the two-term complex $R\SHom_{\pi}(\cs,\cq)$ by $T^{\vir}_{\Quot^{l}_{S}(\ce)}$,
and use the identification of the obstruction sheaf
$$\co b=\SExt^{1}_{\pi}(\cs,\cq)\xrightarrow{\sim}\SExt^{2}_{\pi}(\cq,\cq),$$
obtained by applying $\SExt_{\pi}(-,\cq)$ to the universal exact sequence.

\subsection{Localization along canonical curves}

By definition, the universal quotient sheaf $\cq$ is flat with respect to the smooth projective morphism $\pi$, in particular perfect.
Therefore, we obtain in particular a trace map $\Tr: R\SHom(\cq,\cq)\rightarrow \co$, and denote by
\begin{equation*}
	\Tr_{\pi}^{2}=h^{2}(R\pi_{\ast}\Tr):\SExt^{2}_{\pi}(\cq,\cq)\rightarrow R^{2}\pi_{\ast}\co\simeq\co\otimes \H^{2}(\co_{S})
\end{equation*}
the induced map on top degree cohomology sheaves (we refer to \cite{Illusie} for the standard properties of these trace maps).
We then have a canonical map
\begin{equation*}
	\sigma:\H^{0}(\Omega^{2}_{S})\rightarrow\Hom(\co b,\co)
\end{equation*}
taking a $2$-form $\omega\in \H^{0}(\Omega^{2}_{S})$ to the cosection $\sigma(\omega):\co b\rightarrow\co$ given by
\begin{equation*}
	\sigma(\omega)=\int_{S} \omega\cup\Tr^{2}_{\pi}.
\end{equation*}
Moreover, if $C$ denotes the zero scheme of $\omega$, \ref{onea} gives a closed immersion
$$\iota:\Quot^{l}_{C}(\ce\vert_{C})\rightarrow\Quot^{l}_{S}(\ce).$$
We can now prove \Cref{canonical}.

\begin{theorem}\label{explicitquot}
	Let $\omega$  be a $2$-form on $S$ whose zero scheme is a smooth irreducible curve $C$.
	Then 
	\\
	(i) the zero scheme of the cosection $\sigma(\omega)$ can be identified with $\Quot^{l}_{C}(\ce\vert_{C})$; \\
	(ii) the $\sigma(\omega)$-localized virtual fundamental class of $\Quot^{l}_{S}(\ce)$ is given by
	\begin{equation*}
		[\Quot^{l}_{S}(\ce)]^{\vir}_{\sigma(\omega)}=(-1)^l [\Quot^{l}_{C}(\ce\vert_{C})].
	\end{equation*}
	In particular,
	\begin{equation*}
		\iota_{\ast} [\Quot^{l}_{C}(\ce\vert_{C})]=(-1)^{l}[\Quot^{l}_{S}(\ce)]^{\vir}.
	\end{equation*}
\end{theorem}
\begin{proof}
	(i) Let $q$ be a point of $\Quot^l_{S}(\ce)$. By the base change property of the trace map $\Tr^{2}_{\pi}$, the fibre of $\sigma(\omega)$ over $q$ is given by
	$\int_{S} \omega\cup\Tr^{2}_{q}$, where
	$$\Tr^{2}_{q}:\Ext^{2}(\cq_{q},\cq_{q})\rightarrow\H^{2}(\co_{S})$$
	is the global trace map corresponding to $q$. The multiplicativity property of the trace map asserts that the diagram
	\begin{equation*}
		\begin{tikzcd}
			\Ext^{2}(\cq_{q},\cq_{q}) \arrow[r, "\Tr^{2}_{q}"] \arrow[d, swap, "\cup 1\otimes\omega"] & \H^{2}(\co_{S}) \arrow[d, "\cup\omega"]   \\
			\Ext^{2}(\cq_{q},\cq_{q}\otimes\Omega^{2}_{S}) \arrow[r, swap, "\Tr^{2}_{q}(\Omega^{2}_{S})"] & \H^{2}(\Omega^{2}_{S})
		\end{tikzcd}
	\end{equation*}
	is commutative, where $\cup$ denotes the Yoneda pairing. It  implies in particular that under Serre duality
	\begin{equation*}
		\Ext^2(\cq_q,\cq_q)^\vee\simeq\Hom(\cq_{q},\cq_{q}\otimes\Omega^{2}_{S})
	\end{equation*}
	the fibre of $\sigma(\omega)$ over $q$ corresponds to the map $1\otimes\omega:\cq_q\rightarrow\cq_q\otimes\Omega^{2}_{S}$.
	Hence $\sigma(\omega)$ vanishes at $q$ if and only if the support of $\cq_{q}$ is contained in $C$,
	which in turn is equivalent to $q$ being contained in  $\Quot^{l}_{C}(\ce\vert_{C})\subset\Quot^l_{S}(\ce)$.
	
	(ii) It suffices to show that the assumptions of \Cref{explicit} are satisfied. Let $q$ be a point of $\Quot^{l}_{C}(\ce\vert_{C})$.
	Since $C$ is a curve, the subsheaf $\cs_{q}\subset\ce\vert_{C}$ is locally free of rank $r$, the tangent space
	$$\Hom(\cs_{q},\cq_{q})\simeq\H^{0}(\cs_{q}^{\vee}\otimes\cq_{q})$$
	has dimension $lr$, and the obstruction space
	$$\Ext^{1}(\cs_{q},\cq_{q})\simeq \H^{1}(\cs_{q}^{\vee}\otimes\cq_{q})=0.$$
	Hence $\Quot^{l}_{C}(\ce\vert_{C})$ is smooth of dimension $lr$, which is the virtual dimension of $\Quot^{l}_{S}(\ce)$.
	To see that $\Quot^{l}_{C}(\ce\vert_{C})$ is irreducible, one can use induction on $l$ and the flag Quot scheme: we have
	$\Quot^{1}_{C}(\ce\vert_{C})\simeq\bp(\ce\vert_{C})$, and for the induction step one observes that the canonical  map
	$$\Quot^{l,l+1}_{C}(\ce\vert_{C})\rightarrow C\times\Quot^{l}_{C}(\ce\vert_{C})$$
	is the projectivization of the universal subsheaf, which is locally free.
	To show that the intersection of  $\Quot^{l}_{C}(\ce\vert_{C})$ with the smooth locus $\Quot^{l}_{S}(\ce)^{0}$ is nonempty,
	recall \cite{Stark} that a point $q$ of $\Quot^{l}_{S}(\ce)$ is smooth if and only if $\End(\cq_{q})$ has dimension $l$.
	Thus any point  $q$ of  $\Quot^{l}_{C}(\ce\vert_{C})$ represented by a quotient of the form
	$$\ce\vert_{C}\rightarrow\cl\rightarrow\cl\otimes\co_{Z},$$
	where $\cl$ is an invertible sheaf on $C$ and $Z\subset C$ a subscheme of length $l$, lies in the intersection.
\end{proof}

\begin{example}\label{abundance}
\Cref{explicitquot} applies to a large class of surfaces $S$.
If the canonical sheaf $\Omega^{2}_S$ is very ample --- this is satisfied by any smooth complete intersection
$S\subset\bp_{n}$ of type $(d_{1}, \ldots, d_{n-2})$ with $\sum_{i=1}^{n-2}d_{i}\geqslant n+2$ ($n\geqslant 3$) --- then Bertini's
theorem implies that the zero scheme of a generic $2$-form is smooth and irreducible. If $S$ has trivial canonical sheaf,
then by taking $\omega$ to be nonzero, the cosection $\sigma(\omega)$ is surjective, and thus $[\Quot^{l}_{S}(\ce)]^{\vir}=0$;
in contrast, however, the exceptional divisor of a blow up of $S$ is cut out by a $2$-form.
\end{example}

\begin{remark}
	If $C\subset S$ is a smooth curve which is not canonical, then
	$$\iota_{\ast} [\Quot^{l}_{C}(\ce\vert_{C})]=(-1)^{l}[\Quot^{l}_{S}(\ce)]^{\vir}$$
	is in general false.
\end{remark}

Consider now a locally free quotient $\ce\rightarrow\ce^{''}$ with kernel $\ce^{'}$. By \ref{onea} we have a closed immersion
$$\iota:\Quot^{l}_{S}(\ce^{''})\rightarrow\Quot^{l}_{S}(\ce).$$
As observed by the author in \cite{Stark}, the inclusion $\ce^{'}\rightarrow\ce$ induces a regular section of $\ce^{'\vee[l]}$
with zero scheme $\Quot^{l}_{S}(\ce^{''})$, and we have an exact triangle
\begin{equation}\label{eq:secondexact}
	T_{\Quot^{l}_{S}(\ce^{''})}^{\vir}\rightarrow T^{\vir}_{\Quot^{l}_{S}(\ce)}\vert_{\Quot^{l}_{S}(\ce^{''})}\rightarrow \ce^{'\vee[l]}\rightarrow T^{\vir}_{\Quot^{l}_{S}(\ce^{''})}[1]
\end{equation}
in the derived category of $\Quot^{l}_{S}(\ce^{''})$. This led to the following result:
\begin{theorem}[\cite{Stark}]\label{fundamentalclass}
	We have
	\begin{equation*}
		\iota_{\ast}[\Quot^{l}_{S}(\ce^{''})]^{\vir}=e(\ce^{'\vee[l]})\cap[\Quot^{l}_{S}(\ce)]^{\vir}
	\end{equation*}
	for any locally free quotient  $\ce\rightarrow\ce^{''}$ with kernel $\ce^{'}$.
\end{theorem}
\Cref{explicitquot} implies \Cref{fundamentalclass} for generic surfaces $S$:
\begin{corollary}
	Let $C\subset S$ be a smooth canonical curve. Then
	\begin{equation*}
		\iota_{\ast}[\Quot^{l}_{S}(\ce^{''})]^{\vir}=e(\ce^{'\vee[l]})\cap[\Quot^{l}_{S}(\ce)]^{\vir}
	\end{equation*}
	for any locally free quotient  $\ce\rightarrow\ce^{''}$ with kernel $\ce^{'}$.
\end{corollary}

\begin{proof}
	Indeed, we have a commutative diagram of morphisms of Quot schemes
	\begin{equation}\label{eq:Chow}
		\begin{tikzcd}
			\Quot^{l}_{C}(\ce^{''}\vert_{C}) \arrow[r] \arrow[d] & \Quot^{l}_{C}(\ce\vert_{C})  \arrow[d]   \\
			\Quot^{l}_{S}(\ce^{''}) \arrow[r] & \Quot^{l}_{S}(\ce)
		\end{tikzcd}
	\end{equation}
	induced by the diagram
	\begin{equation*}
		\begin{tikzcd}
			\ce \arrow[r] \arrow[d] & \ce^{''}  \arrow[d]   \\
			\ce\vert_{C}  \arrow[r] & \ce^{''}\vert_{C}
		\end{tikzcd}
	\end{equation*}
	of morphisms  of sheaves on $S$ and property \ref{onea}. The image of $[\Quot^{l}_{C}(\ce^{''}\vert_{C})]$ under $$A_{\ast}(\Quot^{l}_{C}(\ce^{''}\vert_{C})\rightarrow A_{\ast}(\Quot^{l}_{S}(\ce^{''}))$$
	is $(-1)^{l}[\Quot^{l}_{S}(\ce^{''})]^{\vir}$ by \Cref{explicitquot}. On the other hand, the image of $[\Quot^{l}_{C}(\ce^{''}\vert_{C})]$ under
	$$A_{\ast}(\Quot^{l}_{C}(\ce^{''}\vert_{C})\rightarrow A_{\ast}(\Quot^{l}_{C}(\ce\vert_{C}))$$
	is $e(\ce\vert_{C}^{'\vee[l]})\cap[\Quot^{l}_{C}(\ce\vert_{C})]$. Indeed, this follows from the same argument used to prove \Cref{fundamentalclass},
	but here the section cutting out $\Quot^{l}_{C}(\ce^{''}\vert_{C})\subset \Quot^{l}_{C}(\ce\vert_{C})$ is automatically regular (as these Quot schemes are smooth).
	By \Cref{compatibility}, the tautological sheaf $\ce\vert_{C}^{'\vee[l]}$ on $\Quot^{l}_{C}(\ce\vert_{C})$ is the restriction of the tautological sheaf $\ce^{'\vee[l]}$ on $\Quot^{l}_{S}(\ce)$,
	and the map 
	$$A_{\ast}(\Quot^{l}_{C}(\ce\vert_{C})\rightarrow A_{\ast}(\Quot^{l}_{S}(\ce))$$
	takes the class $e(\ce^{'}\vert_{C}^{\vee[l]})\cap[\Quot^{l}_{C}(\ce\vert_{C})]$ to
	$(-1)^{l} e(\ce^{'\vee[l]})\cap[\Quot^{l}_{S}(\ce)]^{\vir}$ by \Cref{explicitquot}. The result thus follows from the commutativity of the square (\ref{eq:Chow}).
\end{proof}

\subsection{Universality}

\begin{lemma}\label{auxiliary}
	Let $C\subset S$ be a smooth curve. Then we have an exact triangle of the form
	\begin{equation*}
		T_{\Quot^{l}_{C}(\ce\vert_{C})}\rightarrow T^{\vir}_{\Quot^{l}_{S}(\ce)}\vert_{\Quot^{l}_{C}(\ce\vert_{C})}\rightarrow R\SHom_{\pi}(\cq_{C},\cq_{C}(C))\rightarrow T_{\Quot^{l}_{C}(\ce\vert_{C})}[1]
	\end{equation*}
	in the derived category of $\Quot^{l}_{C}(\ce\vert_{C})$, where $\cq_{C}$ denotes the universal quotient sheaf on $C\times\Quot^{l}_{C}(\ce\vert_{C})$.
\end{lemma}

\begin{proof}
	By definition of the embedding $\Quot^{l}_{C}(\ce\vert_{C})\rightarrow\Quot^{l}_{S}(\ce)$, we have a morphism of  exact sequences of sheaves
	\begin{equation*}
		\begin{tikzcd}
			0 \arrow[r] & \cs  \arrow[r] \arrow[d]  & \ce \arrow[r]  \arrow[d] & \cq \arrow[r] \arrow[d, "\wr"] & 0  \\
			0 \arrow[r] & \cs_{C}  \arrow[r] & \ce\vert_{C} \arrow[r] & \cq_{C}  \arrow[r] & 0 
		\end{tikzcd}
	\end{equation*}
	on $S\times\Quot^{l}_{C}(\ce\vert_{C})$, where the top exact sequence is obtained by restricting  the universal exact sequence on $S\times\Quot^{l}_{S}(\ce)$,
	and the bottom exact sequence is the pushforward of the universal exact sequence on $C\times\Quot^{l}_{C}(\ce\vert_{C})$. As the middle vertical map is given by restriction,
	taking the long exact sequence of  derived pullbacks with respect to the embedding
	$$C\times\Quot^{l}_{C}(\ce\vert_{C})\rightarrow S\times\Quot^{l}_{C}(\ce\vert_{C})$$
	gives an exact sequence of the form
	\begin{equation*}
		0\rightarrow\cq_{C}(-C)\rightarrow\cs\rightarrow\cs_{C}\rightarrow 0,
	\end{equation*}
	where $\cq_{C}(-C)$ is identified with the first derived pullback of the pushforward of $\cq_{C}$ to $S\times\Quot^{l}_{C}(\ce\vert_{C})$. Applying the functor $R\SHom_{\pi}(-,\cq_{C})$ gives
	the required exact triangle, as $\SHom_{\pi}(\cs_{C},\cq_{C})$ is the tangent sheaf of $\Quot^{l}_{C}(\ce\vert_{C})$ and
	$$R\SHom_{\pi}(\cs,\cq_{C})\simeq L\iota^{\ast}R\SHom_{\pi}(\cs,\cq)$$
	by the projection formula and base change.
\end{proof}

\begin{theorem}\label{universality}
	Let $P$ be a polynomial in the Chern classes of tautological sheaves $\cf_{1}^{[l]}, \ldots, \cf^{[l]}_m$ on $\Quot^{l}_{S}(\ce)$, and the virtual tangent bundle $T^{\vir}_{\Quot^{l}_{S}(\ce)}$.
	Then $$\int_{[\Quot^{l}_{S}(\ce)]^{\vir}} P$$ is given by a polynomial, with coefficients depending on $\ce,\cf_{1},\ldots,\cf_{m}$ only through their ranks,
	in the intersection numbers $$c_{1}(\ce\otimes\cf_{1})c_{1}(S), \ldots, c_{1}(\ce\otimes\cf_{m})c_{1}(S), c_{1}(\ce)c_{1}(S), c_{1}(S)^{2}.$$
\end{theorem}
\begin{proof}
	Since $\cf_{1}^{[l]}, \ldots, \cf^{[l]}_m$ and $T^{\vir}_{\Quot^{l}_{S}(\ce)}$ are compatible with \ref{twoa}, we may assume that $\ce$ is globally generated.
	Consider an exact sequence of the form
	\begin{equation}\label{eq:auxses}
		0\rightarrow\ck\rightarrow\co^{\oplus n}\rightarrow\ce\rightarrow 0.
	\end{equation}
	By \Cref{fundamentalclass}, \Cref{compatibility} and (\ref{eq:secondexact}), we have
	$$\int_{[\Quot^{l}_{S}(\ce)]^{\vir}} P=\int_{[\Quot^{l}_{S}(\co^{\oplus n})]^{\vir}} P^{'}$$
	for a polynomial $P^{'}$ in the Chern classes of $\cf_{1}^{[l]}, \ldots, \cf^{[l]}_m, \ck^{\vee[l]}$ and $T^{\vir}_{\Quot^{l}_{S}(\co^{\oplus n})}$.
	As these are compatible with \ref{threea}, we have by the virtual torus localization formula of Graber-Pandharipande \cite{GP}
	\begin{equation*}
	\int_{[\Quot^{l}_{S}(\co^{\oplus n})]^{\vir}} P^{'}=\int_{[S_{n}^{[l]}]^{\vir}} P^{''}
	\end{equation*}
	for a polynomial $P^{''}$ in the Chern classes of $\cf_{1}^{[l]}, \ldots, \cf^{[l]}_m, \ck^{\vee[l]}$ and the tangent bundle of the disconnected surface $S_n=\coprod_{i=1}^{n} S$. By Serre duality, the obstruction sheaf of $S^{[l]}$ can be identified with the dual of the tautological sheaf $\Omega^{2[l]}_{S}$ (see the proof of Lemma 4.1 in \cite{OP}), and it follows from \Cref{smooth} that the virtual class of $S^{[l]}$ is given by
	\begin{equation}\label{eq:virtualclass}
	[S^{[l]}]^{\vir}=(-1)^{l}e(\Omega^{2[l]}_{S})\cap [S^{[l]}].
	\end{equation}
	We can thus apply the universality theorem for tautological integrals over Hilbert scheme of points \cite{EGL}, and using $c(\ck)=c(\ce)^{-1}$ we obtain that the integral
	\begin{equation*}
	\int_{[\Quot^{l}_{S}(\ce)]^{\vir}} P
	\end{equation*}
	is given by a universal polynomial $U$ in the intersection numbers given by $\ce,\cf_{1},\ldots,\cf_{m}$ and $S$. It remains to show that
	only the intersection numbers
	$$c_{1}(\ce\otimes\cf_{1})c_{1}(S), \ldots, c_{1}(\ce\otimes\cf_{m})c_{1}(S), c_{1}(\ce)c_{1}(S), c_{1}(S)^{2}$$
	can occur. Indeed, the polynomial $U$ can be determined by considering surfaces $S$ containing a smooth canonical curve $C$, since there is an abundance of such surfaces $S$ (\Cref{abundance}).
	By \Cref{explicitquot}, \Cref{auxiliary} and \Cref{compatibility}
	\begin{equation*}
	\int_{[\Quot^{l}_{S}(\ce)]^{\vir}} P=\int_{\Quot^{l}_{C}(\ce\vert_{C})}P_{C} 
	\end{equation*}
	for a polynomial $P_{C}$ in the Chern classes of $\cf_{1}\vert_{C}^{[l]}, \ldots, \cf_{m}\vert_{C}^{[l]}$, $R\SHom_{\pi}(\cq_{C},\cq_{C}(C))$, and $\Quot^{l}_{C}(\ce\vert_{C})$. The argument above,
	mutatis mutandis, also shows that $\int_{\Quot^{l}_{C}(\ce\vert_{C})}P_{C}$ is given by a universal polynomial $U_{C}$ in the genus $g$ of $C$ and the intersection numbers
	$$c_{1}(\ce\vert_{C}), c_{1}(\cf_{1}\vert_{C}), \ldots, c_{1}(\cf_{m}\vert_{C}).$$
	As $U=U_{C}$ after
	$$g=1+c_{1}(S)^{2}, c_{1}(\ce\vert_{C})=-c_{1}(\ce)c_{1}(S), c_{1}(\cf_{i}\vert_{C})=-c_{1}(\cf_{i})c_{1}(S),$$
	it suffices to show that $U_{C}$ is a polynomial in $g$ and
	$c_{1}(\ce\vert_{C}\otimes\cf_{i}\vert_{C})$. For simplicity of notation, we take $m=1$ and $U_{C}$ to be homogenous of degree $d$,
	\begin{equation}\label{eq:U}
		U_{C}=\sum_{i+j+k=d} n_{i,j,k}c_{1}(\cf\vert_{C})^{i}c_{1}(\ce\vert_{C})^{j} g^{k}.
	\end{equation}
	By \ref{twoa} and \Cref{compatibility}, the polynomial $U_{C}=U_{C}(\cf\vert_{C},\ce\vert_{C})$ satisfies
	$$U_{C}(\cf\vert_{C}\otimes\cl,\ce\vert_{C})=U_{C}(\cf\vert_{C},\ce\vert_{C}\otimes\cl)$$
	for any invertible sheaf $\cl$ on $C$. Taking the coefficient of $c_{1}(\cl)^{m}c_{1}(\cf\vert_{C})^{i}c_{1}(\ce\vert_{C})^{j} g^{k}$ gives
	\begin{equation*}
		n_{i,j+m,k}=n_{i+m,j,k}\left(\frac{f}{r}\right)^{m}\frac{(i+m)!j!}{(j+m)!i!},
	\end{equation*}
	where $f$ is the rank of $\cf$. In particular
	\begin{equation*}
	n_{i,j,k}=n_{i+j,0,k}\left(\frac{f}{r}\right)^{j}\binom{i+j}{i}.
	\end{equation*}
	Substituting this into (\ref{eq:U}), we obtain
	\begin{equation*}
		U_{C}=\sum_{i=0}^{d}\frac{n_{i,0,d-i}}{r^{i}}c_{1}(\ce\vert_{C}\otimes\cf\vert_{C})^{i} g^{d-i}. \qedhere
	\end{equation*}
\end{proof}
It is clear from the proof that  $c_{1}(\ce)c_{1}(S)$ can only occur if the first Chern classes of the $\cf_{i}$ can be expressed
in terms of the first Chern class of $\ce$.
\begin{remark}
	As pointed out to us by R. Thomas, one can also approach \Cref{universality} as follows.
	Assume for simplicity that  $\ce$ has rank $r=2$; as in the proof above, we may assume that $\ce$ is globally generated.
	By Bertini's theorem, the zero scheme $Z=Z(s)$ of a generic section $s$ of $\ce$ consists of
	$$n=\int_{S} e(\ce)$$
	reduced points, and the Koszul complex of $s$ gives a nontrivial extension
	$$0\rightarrow\co\rightarrow \ce\rightarrow\ci_{Z}\otimes\det(\ce)\rightarrow 0.$$
	By a standard construction, this yields a deformation of $\ce$ into the torsion-free sheaf $\co\oplus\ci_{Z}\otimes\det(\ce)$,
	inducing a deformation of $\Quot_{S}^{l}(\ce)$ into $\Quot^{l}_{S}(\co\oplus\ci_{Z}\otimes\det(\ce))$.
	One could then use \ref{threea} and torus localization, as well as the results of \cite{EGL}, as the Quot scheme $\Quot^{l}_{S}(\ci_{Z})$
	can be identified with the locus of all $[Z']$ in $S^{[n+l]}$ such that $Z'$ contains $Z$.
      \end{remark}

\section{Applications}

\subsection{Virtual Euler characteristic} The topological Euler characteristic of $\Quot^l_{S}(\ce)$ is given by G\"{o}ttsche's formula
\begin{equation*}
\sum_{l=0}^\infty \chi(\Quot^l_{S}(\ce))q^l=\left(\prod_{m=1}^{\infty}\frac{1}{1-q^{m}}\right)^{r\chi(S)}.
\end{equation*}
Indeed, using $\chi(\Quot^l_{S}(\ce))=\chi(\Quot^l_{S}(\co^{\oplus r}))$ --- which follows for instance from \cite{Ricolfi}) ---- and \ref{three},
one can reduce to the $r=1$ case originally considered by G\"{o}ttsche \cite{Go1}.
By analogy with the Chern–Gauss–Bonnet theorem, one can define the virtual Euler characteristic of $\Quot^{l}_{S}(\ce)$ as
\begin{equation*}
\chi^{\vir}(\Quot^{l}_{S}(\ce))=\int_{[\Quot^{l}_{S}(\ce)]^{\vir}} c(T_{\Quot^{l}_{S}(\ce)}^{\vir}),
\end{equation*}
following Fantechi and G\"{o}ttsche \cite{FG}. For these invariants, the generating series is given by a rational function:
\begin{theorem}[Oprea-Pandharipande \cite{OP}]
	We have
	\begin{equation*}
		\sum_{l=0}^\infty \chi^{\vir}(\Quot^l_{S}(\co^{\oplus r}))q^l=\left(\frac{(1-q)^{2r}}{(1-2^r q)^r}\prod_{i<j}\left(1-(x_i-x_j)^{2}\right)\right)^{c_1(S)^2},
	\end{equation*}
	where $x_1,\ldots,x_r$ are the roots of $x^r-q(x-1)^r=0$. 
\end{theorem}
Our results allow us to generalize this formula to $\Quot^{l}_{S}(\ce)$ for arbitrary $\ce$.
\begin{proposition}\label{independence}
	We have 
	$$\chi^{\vir}(\Quot^{l}_{S}(\ce))=\chi^{\vir}(\Quot^{l}_{S}(\co^{\oplus r})).$$
\end{proposition}

\begin{proof}
	Indeed, if $S$ is the disjoint union of surfaces $S_{1}$ and $S_{2}$, then
	$$\Quot^{l}_{S}(\ce)=\coprod_{l_{1}+l_{2}=l}\Quot^{l_{1}}_{S_{1}}(\ce_{1})\times\Quot^{l_{2}}_{S_{2}}(\ce_{2}),$$
	where $\ce_{1}$ and $\ce_{2}$ are the restrictions of $\ce$ to $S_{1}$ and $S_{2}$, respectively. It is clear
	this splitting is compatible with the perfect obstruction theory, and hence 
	\begin{equation*}
		\sum_{l=0}^{\infty}q^{l}\chi^{\vir}(\Quot^{l}_{S}(\ce))=\sum_{l=0}^{\infty}q^{l}\chi^{\vir}(\Quot^{l}_{S_{1}}(\ce_{1}))\sum_{l=0}^{\infty}q^{l}\chi^{\vir}(\Quot^{l}_{S_{2}}(\ce_{2})).
	\end{equation*}
	This multiplicativity property, \Cref{universality}, and a standard cobordism argument \cite{EGL} show that
	\begin{equation*}
		\sum_{l=0}^{\infty}q^{l}\chi^{\vir}(\Quot^{l}_{S}(\ce))=A^{c_{1}(\ce)c_{1}(S)}B^{c_{1}(S)^{2}}
	\end{equation*}
	for power series $A$ and $B$. As the perfect obstruction theory of $\Quot^{l}_{S}(\ce)$
	is invariant under \ref{twoa}, we obtain
	$$\chi^{\vir}(\Quot^{l}_{S}(\ce\otimes\cl))=\chi^{\vir}(\Quot^{l}_{S}(\ce))$$
	for any invertible sheaf $\cl$ on $S$. Then $c_{1}(\ce\otimes\cl)=c_{1}(\ce)+rc_{1}(\cl)$ gives
	$$A^{rc_{1}(\cl)c_{1}(S)}=1,$$
	which implies $A=1$.
\end{proof}

\begin{remark}
It is clear that $\Quot^{l}_{S}(\ce)$ and $\Quot^{l}_{S}(\co^{\oplus r})$ are locally isomorphic.
Mutatis mutandis, our proof shows that any that any tautological integral over $[\Quot^{l}_{S}(\ce)]^{\vir}$ which is invariant under \ref{twoa} --- for example any integral in the Chern classes of the virtual tangent bundle --- depends on $\ce$ only through its rank $r$. 
\end{remark}

\subsection{Segre integrals}
For any locally free sheaf $\cf$ consider the generating series
\begin{equation*}
	S_{\ce}^{\cf}(q)=\sum_{l=0}^{\infty} (-1)^{lr}q^{l}\int_{[\Quot^{l}_{S}(\ce)]^{\vir}} s(\cf^{[l]}).
\end{equation*}
of Segre integrals. As observed by Oprea and Pandharipande \cite{OP} (in a special case, later generalised in \cite{AJLOP}), this Segre series has a remarkable symmetry property, which can be stated as follows.
\begin{proposition}\label{Segre}
	We have 
	$$S_{\ce}^{\cf}=S^{\ce}_{\cf}.$$
\end{proposition}
\begin{proof}
As in the proof of \Cref{independence}, one shows that
	$$S_{\ce}^{\cf}(q)=A^{c_{1}(\ce\otimes\cf)c_{1}(S)^{2}} B^{c_{1}(S)^{2}}$$
	for power series $A$ and $B$ depending only on the ranks of $\ce$ of $\cf$. 
	The combinatorial argument of \cite{AJLOP} in the case $\ce=\co^{\oplus r}$ shows that $A$ and $B$ are invariant
	under exchanging the ranks of $\ce$ and $\cf$.
\end{proof}

Along the same lines, we obtain the following result. (Use \Cref{universality} to reduce to the case $\ce=\co^{\oplus r}$ dealt with in Corollary 1.16 of \cite{OP}.)
\begin{proposition}
	For any invertible sheaf $\cl$ on $S$
	\begin{equation*}
		\sum_{l=0}^{\infty} q^{l}\int_{[\Quot^{l}_{S}(\ce)]^{\vir}} s(\cl^{[l]})=\left(1+p\right) ^{c_{1}(\ce\otimes\cl)c_{1}(S)}\left(\frac{1+(r+1)p}{(1+p)^{r+1}} \right)^{c_{1}(S)^{2}}
	\end{equation*}
	under $q=(-1)^{r+1}p(1+p)^{r}$.
\end{proposition}

\begin{remark}
Our proof of \Cref{Segre} only partially explains the symmetry property $S_{\ce}^{\cf}=S^{\ce}_{\cf}$, through the dependence on $\ce\otimes\cf\simeq\cf\otimes\ce$ provided by \Cref{universality}. It would be interesting to obtain a deeper understanding, and we believe that in this regard, the implications of \Cref{fundamentalclass} are yet to be fully explored. It is also a simple matter to deduce from \Cref{fundamentalclass}  the rationality of descendent series over $[\Quot^{l}_{S}(\ce)]^{\vir}$ from the corresponding result in the $\ce=\co^{\oplus r}$ case; the latter is the main theorem of the paper \cite{JOP}. This, \Cref{independence}, and \Cref{Segre} can also be deduced from a wall-crossing formula of Joyce \cite{Joyce}, see \cite{Bojko}.
\end{remark}

\subsection{Top intersections of Euler classes}

Finally, we compute the top intersection of the Euler class  $e(\cl^{[l]})$.
\begin{proposition}\label{Euler}
	For any invertible sheaf $\cl$ on $S$
	\begin{equation*}
		\sum_{l=0}^{\infty}q^{l}\int_{[\Quot^{l}_{S}(\ce)]^{\vir}}e(\cl^{[l]})^{r}=\left(\frac{1}{1-q}\right)^{c_{1}(\ce\otimes\cl)c_{1}(S)}.
	\end{equation*}
\end{proposition}

\begin{proof}
	As above, 
	\begin{equation*}
		\sum_{l=0}^{\infty}q^{l}\int_{[\Quot^{l}_{S}(\ce)]^{\vir}}e(\cl^{[l]})^{r}=A^{c_{1}(\ce\otimes\cl)c_{1}(S)}B^{c_{1}(S)^{2}}
	\end{equation*}
	for power series $A=A_{r}$ and $B$ in $q$. To evaluate $B$, we may take $\ce=\co^{\oplus r}$ and $\cl=\co$. By pushing forward
	the universal quotient on $S\times\Quot^{l}_{S}(\co^{\oplus r})$, we obtain a section $s$ of $\co^{\oplus r [l]}$ whose
	value $s(q)\in\Hom(\co^{\oplus r},\cq_{q})$ at a point $q$ of $\Quot^{l}_{S}(\co^{\oplus r})$ is given by the quotient
	corresponding to $q$. Hence
	\begin{equation*}
	e(\co^{[l]})^{r}=e(\co^{\oplus r [l]})=0
	\end{equation*}
	for $l\geqslant 1$, and therefore $B=1$. We now show $A_{r}=(1-q)^{-1}$ by induction on $r$. For $r=1$ consider $\ce=\co$: then $S^{[l]}=\Quot^{l}_{S}(\co)$ has virtual class given by (\ref{eq:virtualclass}). Hence
	\begin{equation*}
	\sum_{l=0}^{\infty}q^{l}\int_{[S^{[l]}]^{\vir}}e(\cl^{[l]})=\sum_{l=0}^{\infty}(-q)^{l}\int_{S^{[l]}}e(\cl^{[l]})e(\Omega^{2[l]}_{S})=\left(\frac{1}{1-q}\right)^{c_{1}(\cl)c_{1}(S)},
	\end{equation*}
	which is well-known (see e.g. Exercise 9.23 in \cite{Na}). For $r\geqslant 2$, consider $\ce=\cl^{\vee}\oplus\ce^{'}$ with $\ce^{'}=(\cl^{\vee})^{\oplus(r-2)}\oplus\co$,
	and the standard exact sequence
	$$0\rightarrow\cl^{\vee}\rightarrow\ce\rightarrow\ce^{'}\rightarrow 0.$$
	We then have $c_{1}(\ce\otimes\cl)=c_{1}(\ce^{'}\otimes\cl)=c_{1}(\cl)$ and
	$$\int_{[\Quot^{l}_{S}(\ce)]^{\vir}}e(\cl^{[l]})^{r}=\int_{[\Quot^{l}_{S}(\ce^{'})]^{\vir}}e(\cl^{[l]})^{r-1}$$
	by \Cref{fundamentalclass}. Hence
	$$A_{r}^{c_{1}(\cl)c_{1}(S)}=A_{r-1}^{c_{1}(\cl)c_{1}(S)}$$
	for any $\cl$ on $S$, which implies $A_{r}=A_{r-1}$.
\end{proof}

\noindent 
\textbf{Acknowledgments.} We are indebted to R. Thomas for a number of helpful conversations on virtual fundamental classes. It is plain that this paper would not have seen the light of day without the pioneering work of Oprea and Pandharipande \cite{OP}. This work was partially supported by Imperial College and the Engineering and Physical Sciences Research Council [EP/L015234/1], and also by the Swiss National Science Foundation SNF-200020-182181.

\noindent Department of Mathematics, Imperial College \\
180 Queen's Gate, London SW7 2AZ, England \\ \\
Departement Mathematik, ETH Zürich \\
Rämistrasse 101, 8006 Zürich, Switzerland \\
samuel.stark@math.ethz.ch


{\small
\begin{thebibliography}{9}
\bibitem{AJLOP} Arbesfeld, N., Johnson, D., Lim, W., Oprea, D., Pandharipande, R.: The virtual $K$-theory of Quot schemes of surfaces, J. Geom. Phys. 164, 104-154 (2021)
\bibitem{BF} Behrend, K., Fantechi, B.: The intrinsic normal cone, Invent. Math. 128, 45-88 (1997)
\bibitem{Bojko} Bojko, A.: Wall-crossing for punctual Quot-schemes, arXiv:2111.11102 (2021)
\bibitem{Chang1} Chang, H-L., Kiem, Y-H.: Poincaré invariants are Seiberg–Witten invariants, Geom. Topol. 17, 1149-1163 (2013)
\bibitem{Chang2} Chang, H-L., Li, J., Li, W-P.: Witten's top Chern class via cosection localization, Invent. Math. 200, 1015-1063 (2015)
\bibitem{Do} Donaldson, S. K.: Polynomial invariants for smooth four-manifolds, Topol. 29, 257–315 (1990)
\bibitem{EG} Ellingsrud, G., G\"ottsche, L.: Variation of Moduli spaces and Donaldson invariants under change of polarization, J. Reine Angew. Math. 467, 1-49 (1995)
\bibitem{EGL} Ellingsrud, G., G\"ottsche, L., Lehn, M.: On the cobordism class of the Hilbert scheme of a surface, J. Alg. Geom. 10, 81-100 (2001)
\bibitem{FG} Fantechi, B., G\"ottsche, L.: Riemann–Roch theorems and elliptic genus for virtually smooth schemes, Geom. Topol. 14, 83-115 (2010)
\bibitem{FS} Fintushel, R., Stern, R. J.: The blowup formula for Donaldson invariants, Ann. Math. 143, 529-546 (1996)
\bibitem{Fulton} Fulton, W.: Intersection Theory (2nd ed.), Springer, New York (1994)
\bibitem{Go1} G\"ottsche, L.: The Betti numbers of the Hilbert scheme of points on a smooth projective surface, Math. Ann. 286, 193-207 (1990)
\bibitem{Go2} G\"ottsche, L.:, Modular forms and Donaldson invariants for 4-manifolds with $b_{+}=1$, J. Amer. Math. Soc. 9, 827–843 (1996)
\bibitem{GK} G\"ottsche, L., Kool, M.: Virtual Segre and Verlinde numbers of projective surfaces, arXiv:2007.11631 (2021)
\bibitem{GP} Graber, T., Pandharipande, R.: Localization of virtual classes, Invent. Math. 161, 487-518 (1999)
\bibitem{Grothendieck} Grothendieck, A.: Techniques de construction et théorèmes d'existence en géométrie algébrique IV (Les schémas de Hilbert), Séminaire Bourbaki 221, 249-276 (1961)
\bibitem{JOP} Johnson, D., Oprea, D., Pandharipande, R.: Rationality of descendent series for Hilbert and Quot schemes of surfaces, Sel. Math. New Ser. 27 (2021)
\bibitem{Joyce} Joyce, D.: Enumerative invariants and wall-crossing formulae in abelian categories, eprint arXiv:2111.04694 (2021)
\bibitem{Illusie} Illusie, L.: Complexe Cotangent et Déformations I (Lecture Notes in Math. 239), Springer, Berlin (1971)
\bibitem{KL} Kiem, Y-H., Li, J.: Localizing virtual cycles by cosections, J. Amer. Math. Soc. 26, 1025-1050 (2013)
\bibitem{KL2}  Kiem, Y-H., Li, J.: Quantum singularity theory via cosection localization, J. reine angew. Math. 766, 73-107 (2020)
\bibitem{KM} Kronheimer, P. B., Mrowka, T. S.: Embedded surfaces and the structure of Donaldson's polynomial invariants J. Diff. Geom. 41, 573-734 (1995)
\bibitem{Li} Li, J.: Algebraic geometric interpretation of Donaldson’s polynomial invariants, J. Diff. Geom. 37, 417–466 (1993)
\bibitem{Lim} Lim, W.: Virtual $\chi_{-y}$-genera of Quot schemes on surfaces, J. London Math. Soc. 104, 1300-1341 (2021)
\bibitem{Maulik} Maulik, D., Pandharipande, R., Thomas, R. P.: Curves on K3 surfaces and modular forms, J. Topol. 3, 937-996 (2010)
\bibitem{Moc} Mochizuki, T.: Donaldson Type Invariants for Algebraic Surfaces, Springer, Berlin (2009)
\bibitem{MP} Maulik, D., Pandharipande, R.: New Calculations in Gromov-Witten Theory, Pure Appl. Math. Q. 4, 469-500 (2008)
\bibitem{Na} Nakajima, H.: Lectures on Hilbert schemes of points on surfaces, AMS, Providence (1999)
\bibitem{OG} O'Grady, K.: Algebro-geometric analogues of Donaldson's polynomials, Invent. Math. 107, 351–395 (1992)
\bibitem{OP} Oprea, D., Pandharipande, R.: Quot schemes of curves and surfaces: virtual classes, integrals, Euler characteristics, Geom. Topol. 25, 3425-3505 (2021)
\bibitem{Ricolfi} Ricolfi, A. T.: On the motive of the Quot scheme of finite quotients of a locally free sheaf, J. Math. Pures Appl. 144, 50–68 (2020)
\bibitem{Stark} Stark, S.: On the Quot scheme $\Quot^{l}_{S}(\mathcal{E})$, arXiv:2107.03991 (2021)
\bibitem{VW} Vafa, C., Witten, E.: A strong coupling test of S-duality, Nucl. Phys. B 431, 3-77 (1994)
\bibitem{Wit} Witten, E.: Supersymmetric Yang–Mills theory on a four-manifold, J. Math. Phys. 35, 5101-5135 (1994)
\end{thebibliography}
}
\end{document}